\documentclass[12pt,reqno]{amsart}
\usepackage{amssymb,amsmath,amsthm}
\usepackage{a4}
\usepackage{graphicx}
\usepackage{rotating}
\usepackage{epstopdf, epsfig}
\usepackage{hyperref}
\usepackage{color}
\begin{document}
\newtheorem{theorem}{Theorem}
\newtheorem{corollary}[theorem]{Corollary}
\newtheorem{lemma}[theorem]{Lemma}
\newtheorem*{example}{Example}
\theoremstyle{definition}
\newtheorem*{definition}{Definition}
\theoremstyle{remark}
\newtheorem{remark}{Remark}
\title[Pythagoras, Binomial, and de Moivre]{\bf Pythagoras, Binomial, and de Moivre Revisited Through Differential Equations}
\markright{}
\author{Jitender Singh$^1$} 
\author{Renu Bajaj$^2$}
\address[1]{Department of Mathematics, Guru Nanak Dev University Amritsar, INDIA; {\tt sonumaths@gmail.com}}
\address[2]{Center for Advanced Study in Mathematics, Panjab University Chandigarh, INDIA; \tt rbjaj@pu.ac.in}
\date{}
\maketitle
\begin{abstract}
\noindent The classical Pythagoras theorem, binomial theorem, de Moivre's formula, and numerous other deductions are made using the uniqueness theorem for the initial value problems in linear ordinary differential equations.
\end{abstract}
\parindent=0cm
\section{Background}
The Pythagoras, binomial, and de Moivre theorems are among the most important formulas in mathematics with wide applications. Hundreds of algebraic, geometric, and dynamic proofs of the more than twenty five hundred years old pythagoras theorem are available in the literature (see the excellent review by Loomi \cite{loomi} and also a  collection of over hundred proofs of Pythagoras theorem available online at \url{https://www.cut-the-knot.org/pythagoras/index.shtml}  \cite{cut}), and discovery of its new proofs still continues. As an evidence to this, we mention that Hirschhorn \cite{hirschhorn} and Luzia \cite{luzia} use geometric approach, Zimba \cite{zimba} provides a beautiful trigonometric proof, while Lengv\'arszky \cite{Z} uses integration to prove the Pythagoras theorem.
The classical proofs of the binomial theorem \cite{lanzo,coolidge,fulton} and de Moivre's formula \cite[Ch.~8,~pp.~106--107]{euler}, \cite{lefschetz} using mathematical induction are also well known. In fact several other proofs of the binomial theorem can be found in the literature using combinatorial \cite{weiner}, probabilistic \cite{probab}, and calculus approaches \cite{hwang,singh17}. However, the proofs using ordinary differential equations are not known much. For instance, Staring \cite{staring} settles the Pythagoras theorem by solving a first order nonlinear ordinary differential equation. The binomial theorem is also proved in Ungar \cite[pp.~874,~Eq.~6]{unger} using an addition theorem for linear ordinary differential equations. Recently, Singh \cite{singh} has obtained de Moivre formula using calculus. His proof is based on the fact that if derivative of a complex valued function of a real variable vanishes identically on the  real line, then the function is constant. In this paper we have proved these results using uniqueness theorem for linear initial value problems.

The uniqueness theorem  for the solution of an initial value problem in a system of first order linear differential equations is well known (see \cite[~pp.~229,~Th.~7.4]{coddington}). An equivalent form for $m$th order linear problem is as follows: given a positive integer $m$; $a_1(t)$, $\ldots$, $a_m(t)$, and $b(t)$ as continuous complex valued functions of
the real variable $t$ on an open interval $I$ (either bounded or unbounded) containing a  real number $t_0$,  the $m$th order linear initial value problem
 \begin{equation}\label{ee1}
 \begin{split}
    y^{(m)}(t)&+a_1(t)y^{(m-1)}(t)+\ldots +a_m(t)y(t)=b(t),\\
    y(t_0)&=y_0,~y'(t_0)=y_1,\ldots,y^{(m-1)}(t_0)=y_{m-1}
    \end{split}
\end{equation}
has a unique solution on the interval $I$.
\section{Pythagoras, binomial, and de Moivre}
To establish Pythagoras, binomial, and de Moivre theorems, we use the aforementioned uniqueness of the solution of a first order ($m=1$) initial value problem in linear ordinary differential equations.
\begin{figure}[h!t!b!]
\centering
\includegraphics[width=0.8\textwidth]{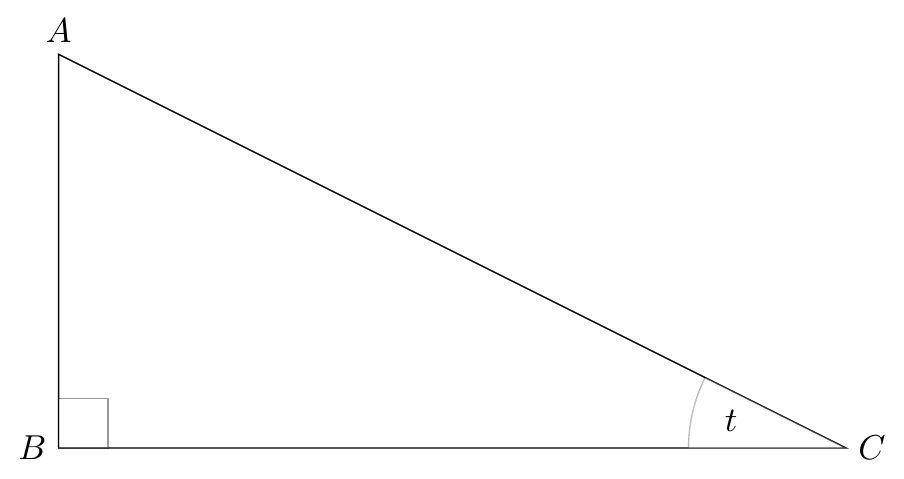}
\caption{Diagram showing the triangle $ABC$ with the angle $ACB=t$ and the angle $ABC=\pi/2$ (in radians).}\label{F1}
\end{figure}
\begin{theorem}[Pythagoras]
Let $ABC$ be the triangle, right angled at its vertex $B$ as shown in Fig.\ \ref{F1}.
If $|AB|$, $|BC|$, and $|AC|$ denote the lengths of the sides $AB$, $BC$, and $AC$ of the triangle $ABC$, respectively, then
\begin{equation}\label{h71}
   |AB|^2+|BC|^2=|AC|^2.
\end{equation}
\end{theorem}
\begin{proof}
Observe that the first order homogenous initial value problem
\begin{equation}\label{c1}
  w'(t)=0,w(0)=0
\end{equation}
is satisfied by $w(t)=\cos^2 t+\sin^2 t-1$ as well as the zero function. By uniqueness of the solution, $w(t)=0$ for all $t$, which proves that
    \begin{equation}\label{h70}
        \cos^2 t+\sin^2t=1,
    \end{equation}
for all real $t$. Now if we take $t$ as the angle $ACB$ in the triangle $ABC$, we have $\sin t=|AB|/|AC|$ and $\cos t=|BC|/|AC|$. Using these in \eqref{h70}, we get  \eqref{h71}.
\end{proof}
\begin{theorem}[Binomial]\label{th3}
For a positive integer $n$ and real $t$, the following holds.
\begin{equation}\label{binom}
(1+t)^n=1+\sum_{k=1}^n\frac{n!}{k!(n-k)!}t^k.
\end{equation}
\end{theorem}
\begin{proof}
Observe that if \eqref{binom} holds for all $t>-1$ then taking limit as $t\rightarrow -1^+$ on both sides of \eqref{binom}, we get the result for $t=-1$ as well.
Also, if \eqref{binom} holds for all $t>-1$ then it holds for all $t<-1$ since if that is the case then we can write  $(1+t)^n=t^{n}(1+\frac{1}{t})^{n}=t^n\sum_{k=0}^n\frac{n!}{k!(n-k)!}\frac{1}{t^k}=\sum_{k=0}^n\frac{n!}{k!(n-k)!}t^{n-k}=\sum_{k=0}^n\frac{n!}{(n-k)!k!}t^{k}$.
Thus, it is sufficient to prove \eqref{binom} for $t>-1$. So, consider the   initial value problem
\begin{equation}\label{p3}
    y'(t)-\frac{n}{1+t}y(t)=0,~y(0)=0,~t>-1,
\end{equation}
which has the zero function as its unique solution for $t>-1$.
Letting $\psi(t)=(1+t)^n-1-\sum_{k=1}^n\frac{n!}{k!(n-k)!}t^k$ for $t>-1$, differentiating it with respect to $t$, and multiplying throughout by $(1+t)$, we get
  \begin{equation*}\label{p4}
\begin{split}
    (1+t)\psi'(t)&=n(1+t)^{n}-(1+t)\sum_{k=1}^n\frac{n!}{k!(n-k)!}kt^{k-1}\\
               &=n(1+t)^{n}-\sum_{k=1}^n\frac{n!}{(k-1)!(n-k)!}(t^{k-1}+t^k)\\
               &=n(1+t)^{n}-n-\sum_{k=2}^n\frac{n!t^{k-1}}{(k-1)!(n-k)!}-\sum_{k=1}^n\frac{n!t^{k}}{(k-1)!(n-k)!}\\
               &=n(1+t)^{n}-n-\sum_{k=1}^{n-1}\frac{n!}{k!(n-k-1)!}t^{k}-\sum_{k=1}^n\frac{n!}{(k-1)!(n-k)!}t^{k}\\
               &=n(1+t)^{n}-n-\sum_{k=1}^{n-1}\frac{n!}{(k-1)!(n-k-1)!}\Bigl(\frac{1}{k}+\frac{1}{n-k}\Bigr)t^{k}-nt^n\\
               &=n\Bigl\{(1+t)^{n}-1-\sum_{k=1}^{n-1}\frac{n!}{k!(n-k)!}t^{k}-t^n\Bigr\}\\
               &=n\Bigl\{(1+t)^{n}-1-\sum_{k=1}^{n}\frac{n!}{k!(n-k)!}t^{k}\Bigr\}=n\psi(t),
    \end{split}
\end{equation*}
which together with the observation that $\psi(0)=0$ shows that $\psi(t)$ satisfies \eqref{p3}. By uniqueness of the solution, we have
$\psi(t)=0$, as desired.
\end{proof}
\begin{theorem}[de Moivre]\label{th2}
For any integer $n$ and a real number $t$, $(\cos{t} + i \sin{t})^n=\cos{nt} +i\sin{nt},$ where $i=\sqrt{-1}$.
\end{theorem}
\begin{proof}
The zero function ($y(t)=0$) satisfies  the initial value problem
\begin{equation}\label{p1}
    y'(t)-i n y(t)=0,~y(0)=0.
\end{equation}
Letting $\varphi(t)=(\cos{t} + i \sin{t})^n-\cos{nt}-i\sin{nt}$, differentiating it with respect to $t$,  and using the product rule of derivatives therein, we get
\begin{equation}\label{p2}
\begin{split}
    \varphi'(t)&=n(\cos{t} + i \sin{t})^{n-1}(-\sin t+i\cos t)+n\sin{nt}-in\cos{nt}\\
               &=in (\cos{t} + i \sin{t})^{n}-in (\cos {nt}+i\sin{nt})=in \varphi(t).
    \end{split}
\end{equation}
From \eqref{p2} and the fact that $\varphi(0)=0$, it follows that $\varphi$ is also a solution of the initial value problem \eqref{p1}. By uniqueness of the solution, $\varphi(t)=0$ for all $t$, that is,  $(\cos{t} + i \sin{t})^n-\cos{nt}-i\sin{nt}=0$, which proves the assertion.
\end{proof}
\begin{remark}
The proofs of binomial and de Moivre formulas in \cite{singh17} and \cite{singh}, respectively, can be translated to fit in the present study. More precisely, both the functions
$$y_1(t)=\frac{1}{(1+t)^n}+\sum_{k=1}^n\frac{n!}{k!(n-k)!}\frac{t^k}{(1+t)^n},~t>-1$$ as well as $$y_2(t)=(\cos t+i\sin t)^{-n}(\cos nt+i\sin nt)$$ satisfy the same initial value problem $$y'(t)=0,~y(0)=1,$$ which at once prove using the uniqueness that $y_1(t)=1$ for $t>-1$ and $y_2(t)=1$ for all $t$. The case $y_1(t)=1$ for $t\leq -1$ can be treated in the same way as in the proof of Theorem \ref{th3}.
\end{remark}
Numerous other applications of the uniqueness theorem can be obtained, such as the following ones, which we leave to the reader for further exploration.
 \section{Euler's formula} The complex exponential $e^{it}$ of a real variable $t$ can be defined as the unique solution of the initial value problem $y'-iy=0;~y(0)=1$. Since the function $\cos t+i\sin t$ also satisfies this initial value problem, it follows by uniqueness that
    \begin{equation}
        e^{it}=\cos t+i\sin t.
    \end{equation}
\section{Trigonometric identities} For any real numbers $c$ and $t$, both the functions $y_1(t)=\sin (t+c)$ as well as  $y_2(t)=\sin t\cos c+\cos t\sin c$ satisfy the second order initial value problem
    \begin{equation*}
        y''(t)+y(t)=0;~y(0)=\sin c,~y'(0)=\cos c.
    \end{equation*}
    By uniqueness, $y_1(t)=y_2(t)$ for all $t$, which proves that
    \begin{equation}
        \sin(t+c)=\sin t\cos c+\cos c \sin t.
    \end{equation}
    Many other identities about the trigonometric functions can be proved the same way after formulating appropriate initial value problems for them.
\section{Sum of geometric progression}
    The initial value problem $$y'+\frac{1}{t-1}y=\frac{nt^{n-1}}{t-1};~y(0)=1,~t<1$$ is satisfied by both the functions $y_1(t)=1+\sum_{k=1}^{n-1}t^k$ as well as $y_2(t)=\frac{1-t^n}{1-t}$. By uniqueness, $y_1(t)=y_2(t)$ for $t<1$. For $t>1$,  $\frac{1}{t}<1$;  so we have
    $1+\sum_{k=1}^{n-1}t^k=$ $t^{n-1}y_1(1/t)=$ $t^{n-1}y_2(1/t)=\frac{1-t^n}{1-t}$. Thus, we have
    \begin{equation}\label{m1}
     1+\sum_{k=1}^{n-1}t^k=\frac{1-t^n}{1-t}~\text{for all}~ t\neq 1.
    \end{equation}
    \begin{remark}\label{r1}
        Differentiating both sides of \eqref{m1} with respect to $t$ and multiplying by $t$, we get $\sum_{k=1}^{n}(k-1)t^{k-1}=t\frac{d}{dt}\Bigl(\frac{1-t^n}{1-t}\Bigr)$ for all $t\neq 1$.
    \end{remark}
\section{Sum of arithmetic-geometric progression}
 Given three real numbers $a$, $t$, and $r\neq 1$, the formula for the sum of arithmetic-geometric progression
       \begin{equation}\label{h1}
       \sum_{k=1}^{n}\{a+(k-1)t\}r^{k-1}=a\frac{1-r^n}{1-r}+tr\frac{1-nr^{n-1}+(n-1)r^{n}}{(1-r)^2},
       \end{equation}
     can be verified through the initial value problem
      \begin{equation*}\label{h12}
        y''(t)=0;~y(0)=a\frac{1-r^n}{1-r},~y'(0)=r\frac{d}{dr}\Bigl(\frac{1-r^n}{1-r}\Bigr),~r\neq 1,
      \end{equation*}
 which is satisfied by each of the left hand side and right hand side of \eqref{h1} for all $t$ (see Remark \ref{r1} for the initial condition $y'(0)$).

 Taking limit as $r\rightarrow 1$ on both sides of \eqref{h1} and solving the limit in the right hand side using L'Hospital rule, we get
 \begin{equation}
    \sum_{k=1}^{n}\{a+(k-1)t\}=\frac{n}{2}\{2a+(n-1)t\},
 \end{equation}
    which provides an alternative proof of the formula for the sum of an arithmetic progression to $n$ terms, which otherwise is often done using the famous Gauss's clever trick.
       \section{Sum of sine and cosine series} Given a real number $a$ and $0<t<2\pi$, one can verify the formulas
            \begin{equation}\label{h2}
            \begin{split}
            \sum_{k=0}^{n-1}\cos(a+kt)&=\frac{\cos\{a+(n-1)t/2\}\sin\{nt/2\}}{\sin (t/2)},\\
       \sum_{k=0}^{n-1}\sin(a+kt)&=\frac{\sin\{a+(n-1)t/2\}\sin\{nt/2\}}{\sin (t/2)}
       \end{split}
       \end{equation}
       as follows.  Here, both the functions $y_1(t)=e^{ia}\sum_{k=1}^{n}e^{i(k-1)t}$ as well as $y_2=e^{ia}\frac{e^{int}-1}{e^{it}-1}$ satisfy the initial value problem
       \begin{equation*}\label{h6}
        y'+\frac{ie^{it}}{e^{it}-1}y=\frac{ie^{i(a+n)t}}{e^{it}-1},~y(\pi)=e^{ia}\frac{1-(-1)^{n}}{2},~0<t<2\pi.
       \end{equation*}
       By uniqueness, $y_1(t)=y_2(t)$ for all $0<t<2\pi$, which on comparing real and imaginary parts recover the two formulas in \eqref{h2}.

The present approach  highlights the importance of the uniqueness theorem of initial value problems of linear ordinary differential equations in  obtaining classical results such as the present ones. The uniqueness theorem may therefore be introduced to the students at schools and colleges at the level, where the fundamental theorem of calculus and differential equations are introduced to them.
The present technique may be useful for stimulating the student who is familiar with ordinary differential equations. Such an enthusiast will enjoy these entertaining applications.



\begin{thebibliography}{1}
\bibitem{loomi} E.\ S.\ Loomi, \textit{The Pythagorean Proposition}, NCTM, (1968).
\bibitem{cut} A.\ Bogomolny, Pythagorean theorem and its many proofs from \textit{Interactive Mathematics Miscellany and Puzzles}, \url{https://www.cut-the-knot.org/pythagoras/index.shtml}.
\bibitem{hirschhorn} M.\ D.\ Hirschhorn, Pythagoras' theorem, \textit{The Math. Gazette} \textbf{92}, (2008), 565.
\bibitem{luzia} N.\ Luzia, A  proof of the Pythagorean theorem after Descartes, \textit{Amer. Math. Month.} \textbf{123}, (2016), 386.
\bibitem{zimba} J.\ Zimba, On the possibility of trigonometric proofs of the pythagorean theorem, \textit{Forum Geometricorum} \textbf{9}, (2009), 275--278.
\bibitem{Z} Z.\ Lengv\'arszky, Pythagoras by integral, \textit{Amer. Math. Month.} \textbf{122}, (2015), 792.
\bibitem{lanzo} G.\ Lanzo, Note on the binomial theorem, \textit{The Analyst}, \textbf{1}, (1874), 177--178.
\bibitem{coolidge} J.\ L.\ Coolidge, The story of the binomial theorem, \textit{Amer. Math. Month.} \textbf{56}, (1949), 147--157.
\bibitem{fulton} C.\ M.\ Fulton, A simple proof of the binomial theorem, \textit{Amer. Math. Month.} \textbf{59}, (1952), 243--244.
\bibitem{euler} L.\ Euler, \textit{Introductio in analysin infinitorum}, Bosquet, Lausanne, 1748. English translation by John Blanton, Springer, New York, 1988 and 1990.
\bibitem{lefschetz} S.\ Lefschetz, A direct proof of de Moivre's formula, \textit{Amer. Math. Month.} \textbf{23}, (1916), 366--368.
\bibitem{weiner} L.\ M.\ Weiner, A direct proof of the binomial theorem, \textit{Math. Teacher} \textbf{48}, (1955), 412.
\bibitem{probab} A.\ Rosalsky, A simple and probabilistic proof of the binomial theorem, \textit{Amer. Statist.}, \textbf{61}, (2007), 161--162.
\bibitem{hwang} L.\ C.\ Hwang, A simple proof of the binomial theorem using differential calculus, \textit{Amer. Statist.}, \textbf{63}, (2009), 43--44.
\bibitem{singh17} J.\ Singh, Another proof of the binomial theorem, \textit{Amer. Math. Month.} \textbf{124}, (2017), 658.
\bibitem{staring} M.\ Staring, The Pythagorean proposition: A proof by means of calculus, \textit{Math. Magazine} \textbf{69}, (1996), 45--46.
\bibitem{unger} A.\ Ungar, Addition theorems in ordinary differential equations, \textit{Amer. Math. Month.} \textbf{94}, (1987), 872--875.
\bibitem{singh} J.\ Singh, A noninductive proof of de Moivre's formula, \textit{Amer. Math. Month.} \textbf{125}, (2018), 80.
\bibitem{coddington} E.\ A.\ Coddington, R.\ Carlson, \textit{Linear Ordinary Differential Equations}, SIAM, Philadelphia, (1997).
\end{thebibliography}
\end{document}